\documentclass[reqno,12pt]{article}
\usepackage{epsfig}

\usepackage{amsmath,amssymb,amsthm}
\usepackage{eucal} 
\usepackage{enumerate} 
\usepackage{bbm}
\usepackage{url}

\newcommand{\vect}[1]{\left(\begin{array}{r} #1 \end{array}\right)}

\newtheorem{Theorem}{Theorem}[section]
\newtheorem{Lemma}[Theorem]{Lemma}
\newtheorem{Proposition}[Theorem]{Proposition}

\theoremstyle{definition} 
\newtheorem{Example}[Theorem]{Example}

\newtheorem{Definition}[Theorem]{Definition}

\newcommand  {\NN}{\mathbb{N}}

\newcommand  {\QQ}{\mathbb{Q}}
\newcommand  {\RR}{\mathbb{R}}
\newcommand  {\ZZ}{\mathbb{Z}}

\newcommand{\cL}{\mathcal{L}}
\newcommand{\cM}{\mathcal{M}}

\newcommand  {\<}{\langle}
\renewcommand{\>}{\rangle}

\newcommand  {\GL}  {\operatorname{GL}}

\newcommand  {\Gr}  {\operatorname{Gr}}
\newcommand  {\Ker} {\operatorname{Ker}}

\newcommand  {\IP}  {\operatorname{IP}}

\renewcommand{\int} {\operatorname{int}}

\newcommand  {\Grp} {\operatorname{Grp}}
\newcommand  {\Hilb}{\mathcal{H}}

\title{ 
Maximal lattice free bodies, test sets and the Frobenius problem 
} 
\author{
Anders Jensen, Niels Lauritzen and Bjarke Roune
}

\begin{document}

\maketitle

\begin{abstract}

  Maximal lattice free bodies are maximal polytopes without interior
  integral points. Scarf initiated the study of maximal lattice free
  bodies relative to the facet normals in a fixed matrix. In this
  paper we give an efficient algorithm for computing the maximal
  lattice free bodies of an integral $(d+1)\times d$
  matrix $A$. An important ingredient is a test set for a certain
  integer program associated with $A$. This test set may be computed
  using algebraic methods.

  As an application we generalize the Scarf-Shallcross algorithm for
  the three-dimensional Frobenius problem to arbitrary dimension.  In
  this context our method is inspired by the novel algorithm by
  Einstein, Lichtblau, Strzebonski and Wagon and the Gr\"obner basis
  approach by Roune.
\end{abstract}


\section{Introduction}\label{intro}

We will introduce this paper by relating it to the Frobenius problem.
Let $\NN$ denote the natural numbers and $\ZZ$ the integers.  For
$a_1, \dots, a_n\in \NN$ with $\ZZ a_1+\cdots+\ZZ a_n=\ZZ$, the
complement of the semigroup $\NN a_1 + \cdots + \NN a_n$ in $\NN$ is
finite. Its maximum is denoted $g(a_1, \dots, a_n)$ and called the
\emph{Frobenius number} of $a_1, \dots, a_n$.  Consider as an example,
the semigroup $S$ generated by $6, 10$ and $15$ in $\NN$. The
complement of $S$ is
$$
\{1, 2, 3, 4, 5, 7, 8, 9, 11, 13, 14,  17, 19, 23, 27, 29\}.
$$ and $g(6, 10, 15) = 29$. Computing the Frobenius number is in
general an NP-hard problem, which is not known to be in NP (see
\cite[Appendix A]{Ramirez2005}). This phenomenon is perhaps related to
the naive method of computing the Frobenius number by finding the
first consecutive sequence of $a_1$ natural numbers in the semigroup.
The search for a consecutive sequence can be simplified through the
following classical result due to Brauer and Shockley
\cite{BrauerShockley1962}: let $r_f$ be the smallest natural number of
the form
$$
x_2 a_2 + \cdots + x_n a_n
$$
congruent to a given integer $f$ modulo
$a_1$, where $x_2,\dots, x_n\in \NN$. 
Then
\begin{equation}
g(a_1,\dots, a_n) = -a_1 + \max\{r_f \mid 0 < f < a_1\}\label{eq:BS}
\end{equation}
The reader may find it useful to deduce the classical result
$$
g(a_1, a_2) = a_1 a_2 - a_1 - a_2
$$ due to Sylvester (1884) as a special case.  

In the language of optimization, \eqref{eq:BS} amounts to solving
$a_1-1$ group programs\footnote{The notion of group programs and
(graph) algorithms for solving them go back to Gomory
\cite{Gomory1965}.}.  These group programs may be interpreted as
finding shortest paths from $0$ in a graph with vertices $\ZZ/a_1 \ZZ$
and edges suitably weighted by $a_2,\dots, a_n$. This is basically the
graph algorithm in \cite{Nijenhuis1979} for computing the Frobenius
number (see also \cite{Beihoffer2005}).  The graph algorithms have the
obvious draw-back that they are exponential in the input size in fixed
dimension.

Surprisingly there are
polynomial algorithms in fixed dimension related to the theory of
integral points in convex polytopes. The first such algorithm was
found by Kannan \cite{Kannan1992}. Later a remarkable algorithm
related to rational generating functions was discovered by Barvinok
and Woods \cite{BarvinokWoods2003}. Both of these algorithms involve
deep mathematical insights, but use quite time consuming
operations of polynomial complexity in fixed dimension.

In this paper we present an
algorithm for enumerating maximal lattice free bodies with facet
normals in the rows of a fixed integral matrix. 
A \emph{maximal lattice free body} is a polytope maximal with respect
to having no interior integral points.  
As an application we
generalize the Scarf-Shallcross algorithm \cite{ScarfShallcross1993}
for computing the Frobenius number for three natural numbers.

Let us briefly recall the
beautiful relation between maximal lattice free bodies and the
Frobenius number: given the coprime natural numbers $a=(a_1,\dots,
a_n)$, Scarf and Shallcross introduced the
polytopes
$$
K_b = \{x\in \RR^{n-1} \mid A x \leq b\}
$$ for $b\in \ZZ^n$, where $A$ is an $n\times (n-1)$ matrix with
columns forming a basis of the lattice $\{v\in \ZZ^n \mid a v =
0\}$. The Frobenius number is then given by
\begin{align*}
  g(a_1, \dots, a_n) = &\max\{a  b \mid b\in \ZZ^n, K_b 
\hbox{ maximal lattice free body}\}\\
  &- \sum a_i.
\end{align*}
For $3\times 2$ matrices Scarf and Shallcross gave a very efficient
algorithm for computing the maximal lattice free bodies up to integral
translation and thereby the Frobenius number for $n=3$.

For integral $n\times (n-1)$ matrices our algorithm
generalizes the Scarf-Shallcross algorithm.  In the context of the
Frobenius problem, an algebraic version of our algorithm has led to
record breaking computations \cite{Roune2007}. It is possible to
extend our algorithm to compute the maximal lattice free bodies for
arbitrary integral matrices by working with ideal points on facets as
in \cite{BaranyScarfShallcross1998}. We only treat the simplicial case
in this paper.

This work was originally prompted by an attempt to understand the
novel algorithm of Einstein, Lichtblau, Strzebonski and Wagon
\cite{EinsteinLichtblauStrzebonskiWagon2006} through the geometric
language in \cite{ScarfShallcross1993}. We thank Daniel Lichtblau and
Stan Wagon for several discussions related to
\cite{EinsteinLichtblauStrzebonskiWagon2006}.

We also thank Herb Scarf, Bernd Sturmfels and Rekha Thomas for
enlightening discussions on neighbors and maximal lattice free bodies.
This work was carried out at Institute for Mathematics and Its
Applications (IMA), University of Minnesota. We thank IMA for
providing optimal research conditions during the thematic year on
``Applications of Algebraic Geometry'', 2006/7.
\
\section{Preliminaries}

We begin by recalling a few concepts and a bit of notation from
polyhedral geometry.
First we define the $m$-dimensional vector
$\max(b_1,\dots, b_r)$ as the coordinate-wise maximum of the vectors
$b_1,\dots,b_r\in \RR^m$.
For example
$$
\max
\left(
\vect{-1 \\ 4 \\ 7},
\vect{ 0 \\ 1 \\ 8},
\vect{-5 \\-3 \\ 6}
\right)=
\vect{ 0 \\ 4 \\ 8}.
$$

\subsection{Polyhedra}

A \emph{polyhedron} $P$ in $\RR^n$ is the set of solutions to a finite
number of linear inequalities in $n$ variables: $a_i^t x\leq b_i$ for
$a_i\in \RR^n$ and $b_i\in \RR$, where $i = 1, \dots, m$.  We will use
the notation $P_A(b):= P =\{x\in \RR^n\mid A x \leq b\}$ collecting
the normal vectors $a_1^t, \dots, a_m^t$ in the rows of an $m\times n$
matrix $A$ and the bounds $b_i$ in a vector $b\in \RR^m$. Notice that
some of the inequalities may be redundant in defining $P_A(b)$. If
$a_i^t x \leq b_i$ is redundant we call it an \emph{inactive facet}.

A polyhedron of the form $P_A(0)$ is
called a \emph{polyhedral cone}. A polyhedral cone
is called \emph{pointed} if it does not contain a line i.e.~if
$\Ker(A)=0$. A polyhedron, $P_A(b)$, is called \emph{rational} if the
entries of $A$ and $b$ are rational numbers.

Suppose $A$ is fixed. Then the smallest polyhedron containing
$v_1, \dots, v_r\in \RR^n$ is given by
$$
\<v_1, \dots, v_r\>_A := P_A(\max(A v_1, \dots, A v_r)).
$$ 

\section{Integer programs and test sets}

Consider a cost vector $c\in \RR^n$. We let $\IP_{A,c}(b)$ denote the
integer linear program
$$
\min\{ c^t x \mid x\in P_A(b)\cap \ZZ^n\} = \min\{c^t x\mid A x \leq b,
x\in \ZZ^n\},
$$
where $b\in \RR^m$ and $A\in\RR^{m\times n}$.

Informally a test set for an integer program is a finite set of
integral (test) vectors, such that a feasible solution is optimal if
it cannot be improved by moving in the direction of any of the test
vectors. Here is the precise definition.

\begin{Definition}
A \emph{test set} for the family $\IP_{A, c}=\{\IP_{A,c}(b) \mid b\in
\RR^m\}$ of integer programs is a \emph{finite} set $T\subset \ZZ^n$ such that 
\begin{enumerate}[(a)]
\item
$c^t  v < 0$ for every $v\in T$.
\item
$v_0\in P_A(b)\cap \ZZ^n$ is optimal if and only if
$v_0 + v\not\in P_A(b)$ for every $v\in T$.
\end{enumerate}
\end{Definition}
Test sets always exist when $A$ is an integral (or rational) matrix.
This follows for example from the following result due to Cook,
Gerards, Schrijver and Tardos.

\begin{Theorem}[{\cite[\S 17.3]{Schrijver1986}}]\label{Theorem:Schrijver}
Let $A$ be an integral $m\times n$ matrix, with all
subdeterminants at most $\Delta$ in absolute value, let $b$ be a
column $m$-vector and $c$ a row $n$-vector. Let $z$ be a feasible, but
not optimal, solution of $\max\{c x\mid A x \leq b; \hbox{$x$
  {\rm integral}}\}$. Then there exists a feasible solution $z'$ such that
$c z' > c z$ and $\parallel z - z' \parallel_\infty < n \Delta$.
\end{Theorem}

For irrational matrices the existence of test sets is more
subtle. This issue has been addressed in Scarf's theory of neighbors
\cite{BaranyScarf1998}.
The following simple lemma will be applied in \S \ref{SS3}.

\begin{Lemma}\label{Lemma:unimod}
  Let $U\in \GL_m(\ZZ)$. Then $T\subset \ZZ^m$ is a test set for
  $\IP_{A,c}$ if and only if $U^{-1} T$ is a test set for $\IP_{A U, c U}$.
\end{Lemma}

\subsection{Conversion to group programs}\label{sect:alg}

Suppose that $A$ is an integral $m\times n$ matrix and $b$ an integral
(column) $m$-vector.  The integer program $\IP_{A, c}(b)$ may be
transcribed in the following way. A feasible solution $x\in \ZZ^n$ to
$\IP_{a, c}(b)$ corresponds to a vector $u\in \NN^ m$ with $A x + u =
b$. If we let $\cL$ denote the subgroup of $\ZZ^m$ generated by the
columns of $A$, this means that $\IP_{A, c}(b)$ may be formulated as
the \emph{group program} $\Grp_{\cL, c'}(b)$:
\begin{align}
&\hbox{min\ } (c')^t u\nonumber \\ 
&u\equiv b\pmod \cL\label{grp}\\ 
&u\in \NN^n\nonumber,
\end{align}
where $c'\in \RR^m$ is some cost vector corresponding to $c\in \RR^n$
(recovering $x$ from $A x + u = b$). In complete analogy with $\IP_{A,
  c}(b)$ we have the following definition.

\begin{Definition}\label{defdef}
A test set for the family $\{\Grp_{\cL, c'}(b) \mid b\in \ZZ^m\}$ of
group programs is a finite set $T\subset \cL$ such that
\begin{enumerate}[(a)]
\item\label{defa}
$(c')^t v < 0$ for every $v\in T$.
\item A vector $u_0\in \NN^n$ with $u_0\equiv b\pmod \cL$ is optimal
  if and only if 
$$
u_0 + v\not\in \NN^n
$$ 
for every $v\in T$.
\end{enumerate}
\end{Definition}

We use the term group program (as opposed to lattice program) since the
optimization problem \eqref{grp} concerns optimizing a function over
certain representatives of a coset in the group $\ZZ^n/\cL$.

In the following we recall how test sets for group programs may be
constructed using Hilbert bases for rational cones. This gives in
principle an algorithm for computing the test set alluded to in
Theorem \ref{Theorem:Schrijver}.

\subsection{Hilbert bases and test sets for group programs}

A pointed polyhedral cone $C\subset \RR^n$ carries a natural partial
order $\prec$ given by
$$
x\prec y\iff y - x \in C
$$ 
for $x, y\in C$. If $C$ is rational and $\cL$ is a
subgroup of $\ZZ^n$, then the semigroup
$\cL\cap C$ has finitely many $\prec$-minimal elements (see
\cite[\S16.4]{Schrijver1986}). These elements are called the {\it
Hilbert basis\/} of $\cL\cap C$ and denoted
$\Hilb(\cL\cap C)$. Every element of $\cL\cap C$ is a finite
non-negative integral linear combination of the elements in
$\Hilb(\cL\cap C)$.

The Euclidean space $\RR^m$ decomposes into the $2^m$ orthants
$\{O_j\}_{j=1}^{2^m}$, which are pointed rational cones. For a lattice
$\cL \subset \ZZ^m$ we put
$$
\Gr(\cL) = \bigcup_{j = 1}^{2^m} \Hilb(\cL \cap O_j).
$$
This finite set of vectors is called the \emph{Graver basis} of $\cL$.
It is not too difficult to prove that $\{v\in \Gr(\cL)\mid c^t v <
0\}$ is a test for the family $\{\Grp_{\cL, c}(b) | b\in \ZZ^m\}$ of
group programs. The full Graver basis is a test set for
the bigger family $\{\Grp_{\cL, c}(b)\mid c\in \RR^m, b\in \ZZ^m\}$ of
group programs, where both $b$ and $c$ are allowed to vary.

\subsection{Generic cost vectors and Gr\"obner bases}

Sufficiently generic cost vectors $c\in \RR^m$ may be viewed as term
orders in computational algebra. The term sufficiently generic
includes the case of linearly independent entries over $\QQ$ in $c$.
In practice we want $c^t v \neq 0$ for $v\in \cL$ inside a sufficiently
big ball centered at $0$. Usually the cost vector is made generic by
breaking ties with a term order $\prec$. We will not make this precise
here but refer the reader to the reference given at the end of this
subsection.

 The \emph{lattice ideal} $I_\cL$ associated to $\cL$ is defined as the ideal
generated by the binomials
$$
\{u^{v^+}-u^{v^-}\mid v\in \cL\}\subset \QQ[u_1, \dots, u_m]
$$
where $v\in\ZZ^m$ is decomposed into vectors $v^+, v^-\in \NN^m$ with
disjoint support such that $v = v^+ - v^-$ and 
$u^w$ denotes the monomial $u_1^{w_1}\cdots
u_m^{w_m}$ for
$w = (w_1, \dots,
w_n)\in \NN^m$. In this context we have the following result.

\begin{Theorem}\label{Theorem:GB}
Let $c$ be a sufficiently generic vector in $\RR^m$.
Then a minimal Gr\"obner basis of $I_\cL$ with respect to the
term order given by $-c$ is 
$$
\{u^{v^+} - u^{v^-}\mid v\in T\},
$$ 
where $T\subset \cL$ is a test set for the family $\{\Grp_{\cL, c}(b)
\mid b\in \ZZ^m\}$ of group programs.
\end{Theorem}

This result is so far the most efficient way of computing test sets
for integer programs: the algebraic viewpoint makes it possible to
apply highly optimized algorithms for computing Gr\"obner bases of
lattice ideals. A state of the art implementation is in the program
{\tt 4ti2} [\url{http://www.4ti2.de}], which also contains functions for
computing and manipulating test sets for integer programs.

For further information on the relation between Gr\"obner bases and
test sets we refer the reader to \cite{Thomas1998}.

\section{Maximal lattice free bodies}

A \emph{convex body} is a compact convex subset of $\RR^n$ with
non-empty interior. We call a convex body \emph{lattice free} if its
interior does not contain any integral points. For a beautiful
exposition of lattice free bodies we refer to \S3 in the survey by Lovasz
\cite{Lovasz1988}.

It is known that a maximal lattice free convex body $B$ is a polytope
\cite[Proposition 3.2]{Lovasz1988}. A surprising result due to Bell
and Scarf says that $B$ has at most $2^n$ facets
\cite[\S16.5]{Schrijver1986}. A very useful characterization of
maximal lattice free bodies is contained in the following result.

\begin{Proposition}\label{Proposition:crucial}
  A polytope is maximal lattice free if and only if its interior does
  not contain any integral points and every facet contains an integral
  point in its relative interior.
\end{Proposition}

Clearly, $\ZZ^n$ acts on the set of maximal lattice free bodies in
$\RR^n$ by translation. We let $\cM(A)$ denote the set of maximal
lattice free bodies of the form $P_A(b)$, where $A$ is an $m\times n$
matrix and $b\in \RR^m$ is the right hand side. These maximal lattice
free bodies are the ones with facet normals in the rows of $A$. For
varying $b$, several of the rows in $A$ may define inactive facets.
The $\ZZ^n$-action on maximal lattice free bodies restricts to an
action on $\cM(A)$ by
$$
z + P_A(b) = P_A(b + A z).
$$ 

It is reasonable to identify maximal lattice free bodies which
differ by an integral translation. With this convention we
have the following result.

\begin{Theorem}\label{Theorem:fini}
  Let $A$ be an integral $m\times n$ matrix.  Up to integral
  translation, $A$ has finitely many maximal lattice free bodies i.e.
$$
\cM(A)/\ZZ^n
$$
is a finite set.
\end{Theorem}
\begin{proof}
  Let $P_A(b)\in \cM(A)$. We may assume without loss of generality
  that all facets of $P_A(b)$ are active. Let $a_1, \dots, a_m$ denote
  the rows of $A$. By Proposition \ref{Proposition:crucial} and
  integral translation we may also assume that the facet corresponding
  to $a_1$ contains $0$ in its relative interior. With these
  reductions, we may assume that $b = (b_1, \dots, b_m)^t$ with $b_1 =
  0$ and $b_i > 0$ for $i>1$. We now prove that that there can only be
  finitely many maximal lattice free bodies of the form $P_A(b)$.

  Let $T$ be the test set in Theorem \ref{Theorem:Schrijver}
  associated with the matrix $A$.  Consider the integer program
\begin{equation}\label{eq:ip}
\min\{a_1 x \mid A x \leq b, x\in \ZZ^n\},
\end{equation}
Let $\epsilon > 0$ be sufficiently small.  Since $P_A(b)$ has no
interior integral points, the integer program given by changing the
right hand side in \eqref{eq:ip} according to $b_2 := b_2-\epsilon,
\dots, b_m := b_m-\epsilon$ has $0$ as optimal solution. On the other
hand $0$ is a feasible but not optimal solution for the integer
program in \eqref{eq:ip} when the right hand side is changed according
to $b_2:= b_2-\epsilon, \dots, b_j := b_j, \dots, b_m :=
b_m-\epsilon$, since $P_A(b)$ contains an integral point in the
relative interior of the $a_j$-facet.  By definition of a test
set, there exists $z_j\in T$ with $a_{1} z_j < 0, a_{i} z_j < b_i$ for
$j\neq i$ and $a_{j} z_j = b_j$.  This shows that a certain subset
$\{z_2, \dots, z_r\}$ with $r-1$ elements of $T$ uniquely defines
$P_A(b)$ up to integral translation.  Since there are finitely many
such subsets, this proves the claim.
\end{proof}

\begin{Example}\label{Example:pat}
The $4\times 2$ matrix
$$ 
A=\left(\begin{array}{rr}
1 & 0\\
-1 & 0\\
0 & 1\\
0 & -1
\end{array}\right).
$$
has $\cM(A) = \emptyset$.
\end{Example}

An interesting question is when $\cM(A)\neq \emptyset$? This is known
to hold if $A$ has full rank and every $m\times m$ minor is
non-vanishing. This condition obviously fails in Example \ref{Example:pat}.

\subsection{The simplicial case}

The proof of Theorem \ref{Theorem:fini} shows that $\cM(A)/\ZZ^n$ can
be computed using a test set for the family
$$
\{\IP_{A, a_1}(b)\mid b\in \RR^m\}
$$
of integer programs, where $a_1$ denotes the first row of $A$.  We
will describe an algorithm for doing this in the more tractable case
related to the Frobenius problem: assume that $A$ is an integral
$(d+1)\times d$ matrix of full rank such that $y A = 0$, for some
$y\in \NN_{>0}^{d+1}$. If $a_0, \dots, a_d$ are the rows of $A$ and $y
= (y_0,\dots, y_d)^t$, then $y_0 a_0 +\cdots + y_d a_d = 0$.

These assumptions imply that every $d\times d$ minor of $A$ is
non-zero and every $K_b := P_A(b)$ is either empty, containing one
point or a full dimensional simplex in $\RR^d$.  The advantage of
being in this setting, is that all facets in $A$ are active in a maximal
lattice free body: let
$$ 
F_i(b) = \{x\in K_b \mid a_i x = b_i\}
$$
denote the facets of a maximal lattice free body $K_b$, where $i =
0,\dots, d$. By Proposition \ref{Proposition:crucial} we may find an
integral point in the relative interior of $F_0(b)$. We may therefore
assume by translation that $0$ is in the relative interior of
$F_0(b)$. Unless the matrix is sufficiently generic, there may be
several choices for this translation as the following example shows.

\begin{Example}\label{Example:nongen}
The $\ZZ^2$-orbit of the maximal lattice free triangle
$$
\left\{\vect{x\\y}\in \RR^2 \left| 
\begin{array}{rrrr}
-x & + &2 y &\leq 0\\
x  &  - &3 y &\leq 1 \\
2x &   - &y &\leq 5
\end{array}
\right.
\right\}
$$
has two representatives containing $0$ in the relative interior of
$F_0(b)$:
\vspace{0.5cm}
$$
\epsfig{file=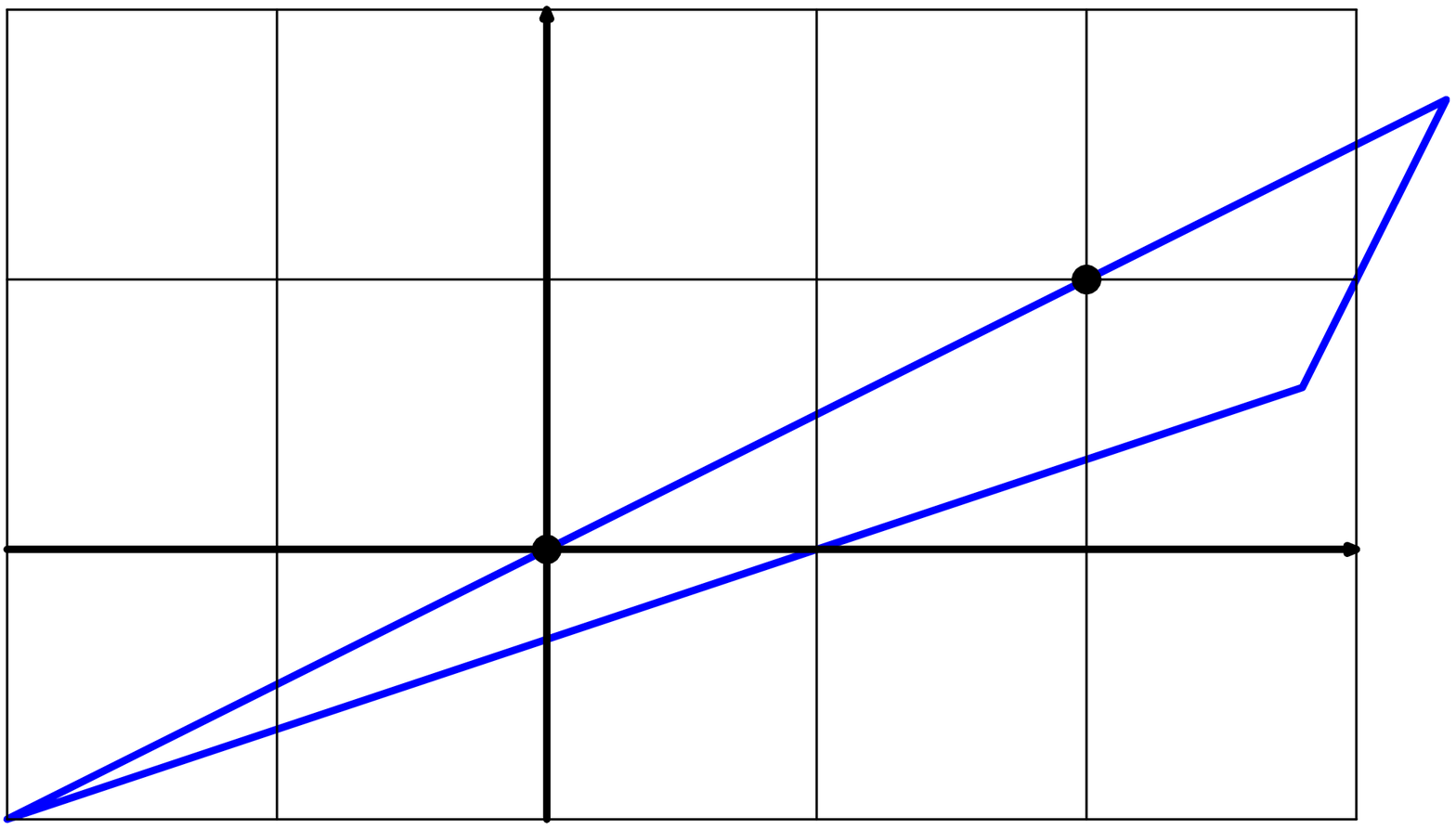,height=7cm}
$$
\vspace{0.5cm}
\end{Example}

\subsection{Test sets and maximal lattice free bodies}

We are interested in computing a well defined representative of a
maximal lattice free body in $\cM(A)/\ZZ^d$. This means, for example,
that we have to choose between the two possible candidates in Example
\ref{Example:nongen}. To make a well defined choice of a representative
we use the following result.

\begin{Lemma}\label{Lemma:unique}
  Every maximal lattice free body is an integral translation of a
  unique maximal lattice free body $K_b$ with the following two properties:
\begin{enumerate}[(a)]
\item 
$b = (0, b_1, \dots, b_d)^t\in \NN^{d+1}$, where $b_1, \dots, b_d>0$.
\item
$0$ is the optimal solution of the integer program
$$
\min\{a_0' z \mid z\in \mathbb{Z}^d, a_1 z\leq b_1-1,
\dots, a_d z \leq b_d-1\},
$$
where
$$
a_0':=a_0 + \epsilon a_1 + \cdots + \epsilon^d a_d
$$
for $\epsilon>0$ small, is a perturbation of the first row vector
$a_0$ in $A$.
\end{enumerate}
\end{Lemma}
\begin{proof}
  If $K_b$ is a maximal lattice free body, then the optimal solutions
  to the integer program
$$
\min\{a_0  z \mid z\in \mathbb{Z}^d, a_1 z\leq b_1-1, \dots, a_d z
\leq b_d-1\}
$$ 
are the integral points in the relative interior of $F_0(b)$.
Therefore they all satisfy $a_0 z = 0$. By perturbing $a_0$ into
$a_o'$ for $\epsilon>0$ small, we identify a unique integral solution
$z_0 \in F_0(b)$ as the optimal solution to the perturbed integer
program
$$
\min\{(a_0') z \mid z\in \mathbb{Z}^d, a_1 z\leq b_1-1, \dots, a_d z
\leq b_d-1\}.
$$
For each maximal lattice free body with $0$ in the relative interior,
$F_0$, of $F_0(b)$ this identifies a unique integral point in $F_0$.
The perturbation of $a_0$ thereby 
identifies unique representatives of the cosets in $\cM(A)/\ZZ^d$.
\end{proof}

The key idea is to use the much smaller test set associated with the
perturbation $a_0'$ in constructing representatives of the maximal
lattice free bodies in $\cM(A)$:

\begin{Theorem}\label{Theorem:MLFB}
  Let $T$ be a test set for the family of integer programs given by
$$
\min\{a_0' z \mid z\in \mathbb{Z}^d, a_1 z\leq b_1, \dots,
a_d z \leq b_d\},
$$
where $b_1,\dots, b_d\in \ZZ$.  If $K_b$ is a maximal lattice free
body as in Lemma \ref{Lemma:unique}, then
there exists a $d$-tuple $(z_1, \dots, z_d)$ of elements in $T$
such that 
\begin{enumerate}[(a)]
\item $a_i z_i > 0$
\item $a_i z_j < a_i z_i$ for $j\neq i$
\item There exists no $z\in T$ with 
$a_k z < a_k z_k$ for all $k=1,\dots,d$,
\item $a_0 z_i \leq 0$.\label{detail}
\end{enumerate}
where $i = 1, \dots, d$. 
\end{Theorem}
\begin{proof}
  This is a straightforward translation of the proof of Theorem
  \ref{Theorem:fini} except for (\ref{detail}): we only have a test
  set with respect to the perturbed cost vector $a_0'$ i.e. we have
  $a_0' z_i < 0$ in (\ref{detail}) but not necessarily $a_0
  z_i < 0$.
\end{proof}

\subsection{Backtracking maximal lattice free
  bodies}\label{Sect:compmlfb}

In this section we outline a backtracking algorithm for enumerating
the unique representatives of the maximal lattice free bodies $K_b$
alluded to in Lemma \ref{Lemma:unique}. Recall that 
$$ 
F_i(b) = \{x\in K_b \mid a_i x = b_i\}
$$
denotes the facet given by the $i$-th row $a_i$ of $A$.

By Theorem \ref{Theorem:MLFB}, $K_b = \<0, v_1, \dots, v_d\>_A$
for a certain $d$-tuple $V=(v_1,\dots, v_d)$ of $T$, such that 
$$
v_i\in F_i(b)\setminus \bigcup_{j=1, j\neq i}^d F_j(b).
$$
Using this observation, there is a simple backtracking algorithm for
generating these specific tuples discarding most of the $d$-tuples
of $T$:
by definition, none of 
$$
\<0, v_1\>_A \subset \cdots \subset \<0, v_1,
\dots, v_d\>_A
$$ 
contain interior points from $T$.  The basic geometric idea is to
assign points from $T$ to $F_1(b), \dots, F_d(b)$ keeping Proposition
\ref{Proposition:crucial} in mind. Define for
$v\in T$ the set
$$
H_i(v) = \{u\in T\mid a_i u < a_i v\},
$$
where $1\leq i \leq d$. Assume inductively that we have constructed a
(partial) tuple $V'=(v_1, \dots, v_i)$ with the property that
$$
v_i \in H_1(v_1)\cap\cdots\cap H_{i-1}(v_{i-1})
$$
for $i\geq 2$ and such that $\<0, v_1,\dots, v_i\>_A$ contains no
interior points from $T$. Then we add $v\in T$
to $V'$ if and only if
\begin{enumerate}[(a)]
\item
$v\in H_1(v_1)\cap\cdots\cap H_i(v_i)$
\item
$\<0, v_1,\dots, v_i, v\>_A$ contains no interior points from $T$.
\end{enumerate}

If this is not possible, $V'$ can never be extended to define a
maximal lattice free body and we backtrack to $V''=(v_1, \dots,
v_{i-1})$ knowing that the extension of $V''$ by $v_i$ leads to a dead
end.

By Theorem \ref{Theorem:MLFB}, we are sure that $B=\<0,
v_1,\dots, v_d\>_A$ is lattice free. However if some $(a_0)^t v_i = 0$,
$B$ may be contained in a strictly larger lattice free body by moving
the facet corresponding to $a_i$. So the backtracking algorithm above
may generate a superset of the maximal lattice free bodies. Note that
it is easy to decide from this superset $S$ whether a lattice free
body is maximal by checking if it is maximal in $S$.

\section{The reverse lexicographic term order}

Consider the usual lexicographic term order $\prec_{lex}$ on vectors in
$\ZZ^n$ starting from the left most coordinate and moving right. Then
the reverse lexicographic term order (for vectors in the subgroup
annihilated by a positive vector $y\in \NN^n$) is given by
$u \prec v \iff -u \prec_{lex} -v$.

The perturbation
$$
a_0' = a_0 + \epsilon a_1 + \cdots + \epsilon^d a_d
$$ 
considered in Lemma \ref{Lemma:unique} has very specific algebraic
meaning.  The (minimal) test set for this perturbation corresponds to
a reverse lexicographic Gr\"obner basis for the ideal $I_\cL$, where
$\cL$ is the lattice generated by the columns of $A$. This can be read
off from the translation (see \S\ref{sect:alg}) of the integer program
$$
\min\{a_0' z \mid A z \leq b\}
$$ 
into the group program
$$ \min\{-u_0 - \epsilon u_1 - \cdots - \epsilon^d u_d \mid u\in
\NN^{d+1}, u \equiv b\pmod {\cL}\}.
$$ 
When $\epsilon$ is infinitesimally small, then 
$$
u\prec v\iff (-1, -\epsilon, -\epsilon^2, \dots, -\epsilon^d)(v -u) > 0
$$
for $u, v\in \ZZ^n$.

\section{The Scarf-Shallcross algorithm}\label{sect:SS}

In this section we explain how our algorithm specializes to the
Scarf-Shallcross algorithm for $3\times 2$-matrices. First we review
the connection between maximal lattice free bodies and the Frobenius
problem given in \cite{ScarfShallcross1993}.  

Consider the Frobenius problem given by $a = (a_1, \dots, a_n)$, where
$a_1, \dots, a_n$ are coprime positive integers. We collect a $\ZZ$-basis of $\cL=\{v\in \ZZ^n \mid a v = 0\}$
in the columns of the $n\times (n-1)$ matrix $A$. 
The key observation is the following.

\begin{Theorem}[Scarf and Shallcross]\label{Theorem:SS}
  Let $b\in\ZZ^n$. The integer $a b$ is representable as a
  non-negative integral linear combination of $a_1, \dots, a_n$ if and
  only if $K_b\cap \ZZ^{n-1} = \emptyset$,
where
$$
K_b = \{x\in \RR^{n-1} \mid A x\leq b\}.
$$
\end{Theorem}
\begin{proof}
  If $a b = a u$ for $u\in \NN^n$, then $b-u\in\cL$.
  Therefore $b - u = A y$ for some $y\in \ZZ^{n-1}$ and $K_b$ contains
  the integral point $y$. If $K_b$ contains an integral point $x$,
  then $u = b-A x\in \NN^{n}$ satisfies $a u = a b$.
\end{proof}

Theorem \ref{Theorem:SS} shows that 
$$
g(a_1, \dots, a_n) = \max\{a  b\mid b\in \ZZ^n\hbox{ and }K_b\cap
\ZZ^{n-1}=\emptyset\}.
$$
In terms of maximal lattice free bodies we have
\begin{align*}
  g(a_1, \dots, a_n) = &\max\{a b\mid b\in \ZZ^n\hbox{ and }K_b \hbox{ is a maximal lattice free body }\}\\
  & - \sum a_i.
\end{align*}

\subsection{Computing the Frobenius number}

In the notation of the beginning of \S\ref{sect:SS}, note that if
$$
K_c = z + K_b
$$
for $z\in \ZZ^{n-1}$ and $b, c\in \ZZ^n$, then $a c = a b$. This implies that
$$
g(a_1, \dots, a_n) = \max\{a b \mid K_b + \ZZ^{n-1}\in \cM(A)/\ZZ^{n-1}\} - \sum a_i
$$
This means that the algorithm in \S \ref{Sect:compmlfb} may be used to
compute the Frobenius number.  The output from this algorithm may
contain non-maximal lattice free bodies along with the representatives
of the maximal ones. From the point of view of computing the Frobenius
number, it is a computational advantage to simply compute the maximum
above using all these bodies.

\subsection{$n=3$}\label{SS3}

In this last subsection we explain how the Scarf-Shallcross algorithm
for computing Frobenius numbers in dimension three may be viewed in
the framework of this paper. We will argue that the reduction
procedure in \cite{ScarfShallcross1993} actually computes a test set
for a specific integer program. The Scarf-Shallcross algorithm is
highlighted in Example \ref{example:ss} of this paper. To get a
feeling for the algorithm the reader is encouraged to start with this
example.

Let $a=(a_1, a_2, a_3)$ be three coprime positive integers and put
$\cL = \{v\in \ZZ^3\mid a v = 0\}$. Then $\cL$ has a particularly
favorable $\ZZ$-basis (see \cite{ScarfShallcross1993}) given by $\cL =
\ZZ u + \ZZ v$, where
$$
u =
\left(\begin{array}{r} 
-\gamma \\ 
\lambda a_1 \\ 
-\mu a_1
\end{array}
\right)
\hbox{\ \ and\ \ }
v = 
\left(\begin{array}{r} 
0 \\
-a_3/\gamma\\
a_2/\gamma
\end{array}
\right),
$$ for $\gamma = \gcd(a_2, a_3)$ and $\gamma = \lambda a_2- \mu a_3$,
where $\lambda, \mu\in \ZZ$ with $0\leq \mu < a_2/\gamma$ and $0 <
\lambda \leq a_3/\gamma$.  By our algorithm we need to compute a test
set for the family of integer programs given by
\begin{equation}\label{eq:ssip}
\min \{-x \mid 
\left(
\begin{array}{rr}
\lambda a_1 & -a_3/\gamma\\
-\mu a_1  & a_2/\gamma y
\end{array}
\right)
\vect{x \\ y} \leq
\vect{b_1 \\ b_2}, x, y\in \ZZ\},
\end{equation}
 where $b_1, b_2\in \ZZ$, to find 
the maximal lattice free bodies $P_A(b)$ for the $3\times 2$ matrix
$$
A=
\left(
\begin{array}{rr}
-\gamma & 0 \\
\lambda a_1 & -a_3/\gamma\\ 
-\mu a_1 & a_2/\gamma
\end{array}
\right).
$$
The Scarf-Shallcross algorithm (\cite{ScarfShallcross1993}, \S
3) finds a unimodular matrix $U$ (see Lemma \ref{Lemma:unimod})
transforming the integer programming problem \eqref{eq:ssip} into
$$
\min\{c^t v \mid B v \leq b, v\in \ZZ^2\},
$$
where 
\begin{align*}
&c_1 < 0, c_2 \leq 0\\ 
&B_{11}>0, B_{12} < 0\\
&B_{21}\leq 0, B_{22}>0
\end{align*}
with $B_{11} + B_{12} \geq 0$ and $B_{21} + B_{22} > 0$. For this
integer programming problem, a test set is $\{e_1, e_2, e_1 + e_2\}$.
According to Lemma \ref{Lemma:unimod}, $\{U^{-1} e_1, U^{-1} e_2, U^{-1}
(e_1 + e_2)\}$ is a test set for the original problem.

\begin{Example}\label{example:ss}
Suppose we wish to compute $g(12, 13, 17)$ using the Scarf-Shallcross
algorithm.  Then
$$
\mathcal{L} = \{(x, y, z) \mid 12 x + 13 y +17z = 0\} = 
\mathbb{Z}
\left(
\begin{array}{r} 
-1 \\ 48 \\ -36
\end{array}
\right) 
+
\mathbb{Z}
\left(
\begin{array}{r} 
0 \\ -17 \\ 13
\end{array}\right).
$$
The following shows the unimodular transformation $U$:
$$
\left(\begin{array}{rr} 
-1 & 0\\
48 & -17\\
-36 & 13
\end{array}\right),
\left(\begin{array}{rr} 
-1 & 0\\
14 & -17\\
-10 & 13
\end{array}\right),
\left(\begin{array}{rr} 
-1 & -1\\
14 & -3\\
-10 & 3
\end{array}\right),
\left(\begin{array}{rr} 
-4 & -1\\
5 & -3\\
-1 & 3
\end{array}\right).
$$ At each step we add a non-negative of one column to the other
preserving the sign pattern for the matrix $A$ we wish to reach. We
stop when this is not possible (as in the last matrix in the
example). In this particular example we see that every maximal lattice
free body is an integral translation of $K_{b_1}$ or $K_{b_2}$, where
$$
b_1 = \max\left(
\left(\begin{array}{r} 
0\\
0\\
0
\end{array}\right),
\left(\begin{array}{r} 
-4\\
5\\
-1
\end{array}\right),
\left(\begin{array}{r} 
0\\
2\\
2
\end{array}\right)
\right)
=
\left(\begin{array}{r} 
0\\
5\\
2
\end{array}\right)
$$
and
$$
b_2 = \max\left(
\left(\begin{array}{r} 
0\\
0\\
0
\end{array}\right),
\left(\begin{array}{r} 
-1\\
-3\\
3
\end{array}\right),
\left(\begin{array}{r} 
0\\
2\\
2
\end{array}\right)
\right)
=
\left(\begin{array}{r} 
0\\
2\\
3
\end{array}\right).
$$
The Frobenius number is
$$ g(12, 13, 17) = \max(-12 + 4\cdot 13 + 1\cdot 17, -12 + 1\cdot 13 +
2\cdot 17) = 57.
$$
\end{Example}

\bibliographystyle{amsplain}

\providecommand{\bysame}{\leavevmode\hbox to3em{\hrulefill}\thinspace}
\providecommand{\MR}{\relax\ifhmode\unskip\space\fi MR }
\providecommand{\MRhref}[2]{%
  \href{http://www.ams.org/mathscinet-getitem?mr=#1}{#2}
}
\providecommand{\href}[2]{#2}

\vspace{0.5cm}
\noindent
{\tt 
anders@soopadoopa.dk, niels@imf.au.dk\\
Institut for Matematiske Fag\\
Aarhus Universitet\\
DK-8000 \AA rhus C\\
Denmark.
}
\end{document}